\documentclass[a4paper,10pt]{amsart}
\usepackage[top=0.8in, bottom=0.8in, left=1.25in, right=1.25in]{geometry}
\usepackage{amsmath}
\usepackage{amsfonts}
\usepackage{amssymb}
\usepackage{graphicx}
\usepackage{mathrsfs}
\usepackage{pb-diagram}
\usepackage{epstopdf}
\usepackage{CJK}

\usepackage{amscd,amsthm,curves,enumerate,latexsym,multibox}
\usepackage{xypic}
\usepackage[all]{xy}
\usepackage{epstopdf}
\newtheorem{theorem}{Theorem}[section]
\newtheorem{lemma}[theorem]{Lemma}
\newtheorem{proposition}[theorem]{Proposition}
\newtheorem{corollary}[theorem]{Corollary}
\theoremstyle{definition}
\newtheorem{definition}[theorem]{Definition}
\newtheorem{remark}[theorem]{Remark}
\newtheorem{example}[theorem]{Example}

\begin{document}
\title[Mukai flops and Pl\"{u}cker formulas for HK]{Mukai flops and Pl\"{u}cker type formulas \\ for hyper-K\"{a}hler manifolds}
\author{Yalong Cao}
\address{The Institute of Mathematical Sciences and Department of Mathematics, The Chinese University of Hong Kong, Shatin, Hong Kong}
\email{ylcao@math.cuhk.edu.hk}

\author{Naichung Conan Leung}
\address{The Institute of Mathematical Sciences and Department of Mathematics, The Chinese University of Hong Kong, Shatin, Hong Kong}
\email{leung@math.cuhk.edu.hk}

\maketitle
\begin{abstract}
We study the intersection theory of complex Lagrangian subvarieties inside holomorphic symplectic manifolds. In particular, we study their behaviour under Mukai flops and give a rigorous proof of the Pl\"{u}cker type formula for Legendre dual subvarieties written down by the second author before. Then we apply the formula to study projective dual varieties in projective spaces.
\end{abstract}


\section{Introduction}
From Weinstein's point of view \cite{weinstein0, weinstein}, everything in the symplectic world is a Lagrangian submanifold. The classical intersection theory of Lagrangian submanifolds is defined by perturbing Lagrangians to have transversal intersections. Fukaya, Oh, Ohta and Ono introduced quantum corrections in terms of counting holomorphic disks bounding Lagrangians and defined their Floer theory \cite{fooo}. It was later generalized by Akaho and Joyce \cite{akahojoyce} to immersed Lagrangian submanifolds (see also Alston and Bao \cite{alstonbao} for some exact immersed cases).

The marriage of (complex) algebraic geometry and symplectic geometry produces an interesting subject which concerns properties of hyper-K\"{a}hler manifolds (or general holomorphic symplectic manifolds).
Hyper-K\"{a}hler manifolds, which have $\mathbb{S}^{2}$-twistor family of complex structures, form a particular type of even dimensional Calabi-Yau manifolds \cite{yau} and are interesting in geometry \cite{donagimarkman, hitchin, huy, leung1, markman, mukai}, string theory and gauge theory \cite{yauzaslow, bea, bleung2, bleung, leeleung, chenxi, kr, dt, cao, caoleung, caoleung3, caoleung2, caoleung5}.

In hyper-K\"{a}hler geometry, complex Lagrangian submanifolds are natural objects and serve as an important source of special Lagrangian submanifolds defined by Harvey and Lawson \cite{harveylawson}.
For complex (algebraic) Lagrangians, their intersection theory can be defined by the normal cone construction \cite{fulton} which does not perturb subvarieties to generic positions and has the advantage for actual computations.
Meanwhile, as for quantum corrections (i.e. holomorphic disks), one can show that they do not exist for generic complex structures in the twistor family (at least when Lagrangians intersect cleanly, see e.g. \cite{leung1}). Because of this, we will concentrate on the classical intersection numbers of
complex Lagrangians in hyper-K\"{a}hler manifolds.

The main purpose of this article is to study the behaviour of intersection numbers of complex Lagrangian subvarieties (more generally, half-dimensional subvarieties) under a birational transformation called the Mukai flop \cite{mukai} of projective hyper-K\"{a}hler manifolds (more generally, even dimensional projective manifolds).

Let $M$ be a $2n$-dimensional projective manifold containing a $\mathbb{P}^{n}$ with $\mathcal{N}_{\mathbb{P}^{n}/M}\cong T^{*}\mathbb{P}^{n}$.
We denote its \emph{Mukai flop} along $\mathbb{P}^{n}$ by
\begin{equation}\phi:M\dashrightarrow M^{+},  \nonumber \end{equation}
which is the composition of the blow up $\widetilde{M}$ of $M$ along $\mathbb{P}^{n}$ and the blow down of $\widetilde{M}$ along the exceptional locus in another ruling. Note that the exceptional locus blows down to a $\mathbb{P}^{n}$ in $M^{+}$ which we denote by $(\mathbb{P}^{n})^{*}\subseteq M^{+}$. If $C\subseteq M$ is a half-dimensional closed irreducible subvariety,
we denote its \emph{strict transformation}\footnote{The strict transformation is often referred as \emph{Legendre transformation} if $M$ is hyper-K\"{a}hler and $C$ is its Lagrangian subvariety.} by
\begin{equation}C^{\vee}\triangleq\overline{\phi(C\backslash \mathbb{P}^{n})}\subseteq M^{+}, \quad \textrm{if}\textrm{ } C\neq\mathbb{P}^{n}, \nonumber \end{equation}
\begin{equation}(\mathbb{P}^{n})^{\vee}\triangleq(-1)^{n}(\mathbb{P}^{n})^{*}\subseteq M^{+}, \nonumber \end{equation}
which satisfies the reflection property $(C^{\vee})^{\vee}=C$.
The behaviour of intersection numbers of half-dimensional subvarieties under Mukai flops is given by the following Pl\"{u}cker type formula.
\begin{theorem}\label{thm 1.1}(Theorem \ref{plucker formula 1}, see also Theorem 23 of \cite{leung1}) ${}$ \\
Let $\phi:M\dashrightarrow M^{+}$ be a Mukai flop along $\mathbb{P}^{n}$ ($n\geq2$) between projective manifolds, then
\begin{equation}C_{1}\cdot C_{2}+\frac{(C_{1}\cdot \mathbb{P}^{n})(C_{2}\cdot \mathbb{P}^{n})}{(-1)^{n+1}(n+1)}
=C_{1}^{\vee}\cdot C_{2}^{\vee}+\frac{(C_{1}^{\vee}\cdot (\mathbb{P}^{n})^{*})(C_{2}^{\vee}\cdot (\mathbb{P}^{n})^{*})}{(-1)^{n+1}(n+1)}
\nonumber \end{equation}
holds for $n$-dimensional closed irreducible subvarieties $C_{1}$, $C_{2}$ of $M$ and their strict transformations $C_{1}^{\vee}$, $C_{2}^{\vee}$
in $M^{+}$.
\end{theorem}
This formula was first written down (without proof) by the second author in \cite{leung1}. We will give it a rigorous proof in this paper and then apply it to the study of the geometry and topology of projective duality.

Now we explain why this formula is of Pl\"{u}cker type. We consider two irreducible subvarieties $S_{i}$ ($i=1,2$) in $\mathbb{P}^{n}$ and their
projective dual $S^{\vee}_{i}\subseteq (\mathbb{P}^{n})^{*}$ in the dual projective space \cite{gkz}. The conormal varieties $C_{S_{i}}\subseteq T^{*}\mathbb{P}^{n}$ are in fact strict transformations of $C_{S^{\vee}_{i}}$ under Mukai flop $T^{*}\mathbb{P}^{n}\dashrightarrow T^{*}(\mathbb{P}^{n})^{*}$, i.e. $C_{S^{\vee}_{i}}=(C_{S_{i}})^{\vee}$. When $S_{1}\pitchfork S_{2}$, the intersection of $C_{S_{1}}$ and
$C_{S_{2}}$ happens inside the zero section of $T^{*}\mathbb{P}^{n}$, we use Theorem \ref{thm 1.1} to deduce
\begin{theorem}(Theorem \ref{plucker formula 2}, Proposition \ref{prop 1.1}, \ref{prop 1.2}) ${}$ \\
Let $S_{1}$, $S_{2}$ be two closed irreducible subvarieties in $\mathbb{P}^{n}$ ($n\geq2$) which intersect transversally and the same holds true
for their dual varieties $S_{1}^{\vee}$, $S_{2}^{\vee}$ in $(\mathbb{P}^{n})^{*}$. Then we have
\begin{equation}
(-1)^{*}\bigg(\chi(S_{1}\cap S_{2})-\frac{\chi\big(S_{1},Eu([S_{1}])\big)\chi\big(S_{2},Eu([S_{2}])\big)}{n+1}\bigg)=
\chi(S^{\vee}_{1}\cap S^{\vee}_{2})-\frac{\chi\big(S^{\vee}_{1},Eu([S^{\vee}_{1}])\big)\chi\big(S^{\vee}_{2},Eu([S^{\vee}_{2}])\big)}{n+1},
\nonumber \end{equation}
where $*=dim S_{1}+dim S_{2}+dim S^{\vee}_{1}+dim S^{\vee}_{2}$, $Eu$ is MacPherson's Euler obstruction and $\chi(-,Eu([-]))$ is the weighted Euler characteristic with respect to constructible function $Eu([-])$.
\end{theorem}
As corollaries, we could apply this formula to determine degrees of projective dual varieties (Corollary \ref{deg of dual var}) which recovers Ernstr\"{o}m's generalized Pl\"{u}cker formulas (see also \cite{ern1, ern2, kleiman, matsui}). We also use it to determine
degree zero Chern-Mather classes, dimensions of projective dual varieties (see Corollary \ref{chern-mather class} and \ref{dim of dual var}).
Because of these applications to the geometry of projective duality, we call the formula in Theorem \ref{thm 1.1} of Pl\"{u}cker type.

The proof of Theorem \ref{thm 1.1} is based on an equivalence $D^{b}(M)\cong D^{b}(M^{+})$ of derived categories of coherent sheaves established by Kawamata \cite{kawamata} and Namikawa \cite{namikawa}. From this point of view, the above equivalence of categories could be regarded as a more general 'Pl\"{u}cker type formula'.

${}$ \\
\textbf{The content of this paper}: In section 2, we recall the intersection theory of complex Lagrangian subvarieties.
In section 3, using derived equivalences established by Kawamata and Namikawa, we deduce a Pl\"{u}cker type formula for intersection numbers of half-dimensional subvarieties in even dimensional projective manifolds under Mukai flops. We also apply it to deduce a Pl\"{u}cker type formula for projective dual varieties in projective spaces which has many applications to the geometry of projective duality. In the final section, we construct examples of complex Lagrangian subvarieties in hyper-K\"{a}hler manifolds, which are motivated by Donaldson-Thomas' higher dimensional gauge theories.

${}$ \\
\textbf{Acknowledgement}:
The first author thanks Garrett Alston, Dominic Joyce and his fellow colleagues Qingyuan Jiang, Ying Xie and Chuijia Wang for helpful discussions.
The work of the second author was substantially supported by grants from the Research Grants Council of the Hong Kong Special Administrative Region, China (Project No. CUHK401411 and CUHK14302714).

\section{Intersection theory of complex Lagrangians in hyper-K\"{a}hler manifolds}
In this section, we recall some standard facts about the intersection theory of complex Lagrangian subvarieties inside holomorphic symplectic manifolds.

\subsection{Some basic notions from hyper-K\"{a}hler geometry}
Let $(M,g)$ be a hyper-K\"{a}hler manifold with complex structures $I,J,K$ which satisfy the Hamilton relation
\begin{equation}I^{2}=J^{2}=K^{2}=IJK=-Id.   \nonumber \end{equation}
Then the twistor family of complex structures on $M$ can be expressed as
\begin{equation}\mathbb{S}^{2}=\{aI+bJ+cK \textrm{ }|\textrm{ } a^{2}+b^{2}+c^{2}=1 \textrm{ }\textrm{and}\textrm{ } a,b,c\in \mathbb{R} \}
\nonumber \end{equation}
Fixing a complex structure $J$, $(M,g,J)$ is a Calabi-Yau manifold which corresponds to the embedding $Sp(n)\subseteq SU(2n)$.
We denote the K\"{a}hler form by $\omega_{J}$. The parallel form
\begin{equation}\Omega_{J}=\omega_{I}-i\omega_{K}  \nonumber \end{equation}
defines a \emph{holomorphic symplectic structure} on $(M,J)$. And the $J$-holomorphic volume form for the Calabi-Yau structure on $(M,g,J)$ is just the top exterior power of $\Omega_{J}$.

A $J$-\emph{complex Lagrangian submanifold} $L\subseteq M$ is a half-dimensional submanifold with
\begin{equation}\Omega_{J}|_{L}=0 \textrm{ }\footnote{Hitchin showed that this already implies $L$ is a $J$-holomorphic submanifold of $M$ \cite{hitchin0}.}.
\nonumber \end{equation}
Similar to the real case, if $L$ is a complex Lagrangian submanifold, we have $\mathcal{N}_{L/M}\cong T^{*}L$.
In particular, for any half-dimensional Fano ($c_{1}>0$) $J$-complex submanifold, one can identify its normal bundle and cotangent bundle (see e.g. \cite{leung1}).

More generally, a \emph{complex Lagrangian subvariety} $L\subseteq M$ is an irreducible locally closed\footnote{i.e. it is the open subset of a closed subset in the Zariski topology.} subvariety whose smooth locus is a complex Lagrangian submanifold. The following result will be useful for the construction of
compact complex Lagrangian subvarieties.
\begin{proposition}\label{prop}
Let $L\subseteq M$ be a complex Lagrangian subvariety in an algebraic symplectic manifold. Then its (Zariski) closure
$\overline{L}\subseteq M$ (with reduced structures) is also a complex Lagrangian subvariety.
\end{proposition}
\begin{proof}
As $L$ is locally closed, it is an open subvariety of $\overline{L}$. The irreducibility of $L$ extends to $\overline{L}$.
We are left to show its smooth locus is a complex Lagrangian submanifold.
Since $T_{x}(\overline{L})_{sm}$ being a complex Lagrangian subspace (i.e. $\Omega_{J}|_{T_{x}(\overline{L})_{sm}}=0$) is a closed condition among points $x\in(\overline{L})_{sm}$, thus $(\overline{L})_{sm}$ is also a complex Lagrangian submanifold (as $L_{sm}\subseteq(\overline{L})_{sm}$ is open and $(\overline{L})_{sm}$ is irrducible).
\end{proof}
We refer to \cite{hitchin0, huy, joyce7, leung1, leung} for more details about hyper-K\"{a}hler geometry.

\subsection{Intersection numbers of complex Lagrangians}
Let $L_{1}$ and $L_{2}$ be two closed complex Lagrangian subvarieties in a projective hyper-K\"{a}hler manifold $M$.
We have the following two types of intersection numbers (defined for any two half-dimensional subvarieties). \\
\emph{Topological intersection}: $L_{1}\cdot L_{2}\triangleq deg([L_{1}]\cup[L_{2}])$, where $[\cdot]$ denotes the Poincar\'{e} dual of the corresponding fundamental class \cite{gh}.  \\
\emph{Algebraic intersection}: $L_{1}\cdot L_{2}$ is the degree of their intersection product (defined by taking normal cone and applying the refined Gysin map \cite{fulton}). These two approaches have their own advantages and coincide with each other (see for instance Corollary 19.2 \cite{fulton}).

Next, we discuss the case when $L_{1}$ and $L_{2}$ are both smooth, following the work of Behrend and Fantechi \cite{behrend}, \cite{bf}, \cite{bf2}. The scheme theoretical intersection $X=L_{1}\cap L_{2}$ then
has a symmetric obstruction theory \cite{behrend}
\begin{equation}\varphi: E^{\bullet}\rightarrow \mathbb{L}_{X},    \nonumber \end{equation}
which can be represented as $E^{\bullet}\simeq [T^{*}M|_{X}\rightarrow T^{*}L_{1}|_{X}\oplus T^{*}L_{2}|_{X}]$ \cite{bf2}.
Then Behrend's result gives
\begin{equation}L_{1}\cdot L_{2}=deg[X,E^{\bullet}]^{vir}=\chi(X,\nu_{X}),    \nonumber \end{equation}
where $[X,E^{\bullet}]^{vir}\in A_{0}(L_{1}\cap L_{2})$ is the virtual cycle \cite{bf}, \cite{lt1} of the perfect obstruction theory $\varphi$,
$\nu_{X}$ is Behrend's constructible function of $X$ defined as MacPherson's Euler obstruction of the intrinsic normal cone of $X$, and
$\chi(X,\nu_{X})=\sum_{n\in\mathbb{Z}}n\chi\{\nu_{X}=n\}$ is the weighted Euler characteristic with respect to $\nu_{X}$. This then implies
that for smooth complex Lagrangian submanifolds $L_{1}$, $L_{2}$, their intersection number \textbf{depends only} on the scheme structure of the intersection $X=L_{1}\cap L_{2}$. In particular, we could use weighted Euler characteristics as the definition of intersection numbers even when the intersection is non-compact.

\subsection{Intersection cohomologies of complex Lagrangians}
The symmetric obstruction theory on the scheme theoretical intersection $X=L_{1}\cap L_{2}$ of two complex Lagrangian submanifolds in fact
can be enhanced to a more refined structure called $(-1)$-shifted symplectic structure in the sense of Pantev, T\"{o}en, Vaqui\'{e} and Vezzosi \cite{ptvv}. Joyce introduced d-critical loci \cite{joyce} as classical truncations of derived schemes with $(-1)$-shifted symplectic structures and showed that for any oriented d-critical loci $(X,s,K^{1/2}_{X,s})$, there exists a perverse sheaf $\mathcal{P}^{\bullet}_{X,s}$ on $X$ whose hypercohomology categorifies the weighted Euler characteristic of $X$, i.e.
\begin{equation}\sum_{i}(-1)^{i}\mathbb{H}^{i}(X,\mathcal{P}^{\bullet}_{X,s})=\chi(X,\nu_{X}).  \nonumber \end{equation}
The orientation $K^{1/2}_{X,s}$ in the Lagrangian intersection case corresponds to a choice of square roots of $K_{L_{1}}|_{X^{red}}\otimes K_{L_{2}}|_{X^{red}}$ over the reduced scheme $X^{red}$ of $X$. In particular, if $L_{i}$ ($i=1,2$) are \emph{orientable} complex Lagrangian submanifolds (i.e. $\exists$ $K^{1/2}_{L_{i}}$ ) \cite{bbdjs}, \cite{bussi}, $X=L_{1}\cap L_{2}$ is orientable.

Other interesting works towards the categorification or quantization of complex Lagrangian intersections include works of Baranovsky and Ginzburg \cite{bg}, Kapustin and Rozansky \cite{kr}, Kashiwara and Schapira \cite{kashiwaresch}, etc.
\begin{remark}
Given two \emph{orientable} complex Lagrangians $L_{1}$, $L_{2}$ in a hyper-K\"{a}hler manifold (with a fixed holomorphic symplectic form), they are real Lagrangians for a $S^{1}$-family of real symplectic structures $\omega_{\theta}$. Now we have two Floer type cohomologies associated with them, i.e. \\
\emph{Joyce}'s version: $\mathbb{H}^{*}(L_{1}\cap L_{2},\mathcal{P}^{\bullet}_{L_{1}\cap L_{2}})$; \\
\emph{Fukaya-Oh-Ohta-Ono}'s version: $HF_{\omega_{\theta}}^{*}(L_{1},L_{2})\triangleq HF_{\omega_{\theta}}^{*}(L_{1},\phi(L_{2}))$ for a symplectic form $\omega_{\theta}$ and a Hamiltonian diffeomorphism $\phi$ such that $L_{1}\pitchfork \phi(L_{2})$ \footnote{We use Novikov field as coefficient to kill torsion, so it coincides with Definition 6.5.39 of \cite{fooo}.}.

When the intersection $L_{1}\cap L_{2}$ is clean, $\mathcal{P}^{\bullet}_{L_{1}\cap L_{2}}$ is a constant sheaf and Joyce's Floer cohomology is simply the singular cohomology  $H^{*}(L_{1}\cap L_{2})$. Then there exists a spectral sequence
\begin{equation}E_{2}=H^{*}(L_{1}\cap L_{2})\Rightarrow HF_{\omega_{\theta}}^{*}(L_{1},L_{2})  \nonumber \end{equation}
converging to $HF_{\omega_{\theta}}^{*}(L_{1},L_{2})$ whose $E_{2}$ page is $H^{*}(L_{1}\cap L_{2})$ with certain coefficient in the Novikov field (see Theorem 6.1.4 of FOOO \cite{fooo} for detail).
In this set-up, for at most one exceptional $\theta\in [0,2\pi)$, there is no $J_{\theta}$-holomorphic disk bounding $L_{1}\cup L_{2}$ \footnote{See e.g. Lemma 13 of \cite{leung1}.}, and the spectral sequence degenerates at $E_{2}$ page. It is an interesting question to extend this picture to singular intersection cases.
\end{remark}

\section{Mukai flops and Pl\"{u}cker type formulas}
In this section, we study the behaviour of intersection numbers of complex Lagrangian subvarieties (more generally, half-dimensional subvarieties) inside projective hyper-K\"{a}hler manifolds (more generally, even dimensional projective manifolds) under the Mukai flop, which
is summarized as a Pl\"{u}cker type formula. Then we apply it to projective dual varieties in projective spaces and give applications to their geometry and topology.

\subsection{Mukai flops and derived equivalences}
We recall the definition of Mukai flops and how derived categories of coherent sheaves behave under these birational transformations between projective manifolds of even dimensions (besides projective hyper-K\"{a}hler manifolds).

We start with a $2n$-dimensional complex projective manifold $M$ which contains an embedded $\mathbb{P}^{n}$ such that $\mathcal{N}_{\mathbb{P}^{n}/M}\cong T^{*}\mathbb{P}^{n}$. We denote the blow up of $M$ along $\mathbb{P}^{n}$ by $\widetilde{M}$ whose exceptional divisor $E=\mathbb{P}(T^{*}\mathbb{P}^{n})\subseteq \mathbb{P}^{n}\times(\mathbb{P}^{n})^{*}$ is the incidence variety. By the adjunction formula, one has $\mathcal{N}_{E/\widetilde{M}}\cong\mathcal{O}(-1,-1)$. Thus $\widetilde{M}$ admits a blow down $\widetilde{M}\rightarrow M^{+}$ to a projective manifold $M^{+}$ whose restriction to $E$ is the projection $E\subseteq \mathbb{P}^{n}\times(\mathbb{P}^{n})^{*}\rightarrow(\mathbb{P}^{n})^{*}$ to the second factor \cite{huy1}. The birational transformation
\begin{equation}\phi:M\dashrightarrow M^{+}  \nonumber \end{equation}
given by the composition of the above blow up and down is called \emph{Mukai flop} along $\mathbb{P}^{n}$ \cite{mukai0, mukai, huy1}.

As normal bundles of $\mathbb{P}^{n}$ and $(\mathbb{P}^{n})^{*}$ are their cotangent bundles, $M$ and $M^{+}$ admit blow down to a variety $\overline{M}$. We denote the fiber product $\widehat{M}=M\times_{\overline{M}}M^{+}$ which has canonical morphisms
\begin{equation}\label{comm diagram}\xymatrix{ & \widehat{M} \ar[dr]^{\pi^{+}} \ar[dl]_{\pi}   \\  M    &  &  M^{+}  }      \end{equation}
to $M$ and $M^{+}$. $\widehat{M}$ is a normal crossing variety with two irreducible components $\widetilde{M}$ and $\mathbb{P}^{n}\times(\mathbb{P}^{n})^{*}$.
\begin{theorem}(Kawamata \cite{kawamata}, Namikawa \cite{namikawa})\label{derived equi of Mukai flops} ${}$ \\
If $n\geq2$, the functor
\begin{equation}\Psi\triangleq\textbf{R}(\pi^{+})_{*}\circ \textbf{L}\pi^{*}:D^{b}(M)\rightarrow D^{b}(M^{+})   \nonumber \end{equation}
is an equivalence of triangulated categories.
\end{theorem}
In particular, given two bounded complexes of sheaves $\mathcal{E}^{\bullet}_{i}$, $i=1,2$, we have isomorphisms
\begin{equation}Ext^{*}_{M}(\mathcal{E}^{\bullet}_{1},\mathcal{E}^{\bullet}_{2})\cong Ext^{*}_{M^{+}}(\Psi(\mathcal{E}^{\bullet}_{1}),\Psi(\mathcal{E}^{\bullet}_{2})),  \nonumber \end{equation}
and an equality
\begin{equation}\label{equality of char}\int_{M} \overline{ch}(\mathcal{E}^{\bullet}_{1})\cdot ch(\mathcal{E}^{\bullet}_{2})\cdot Td(M)=\int_{M^{+}} \overline{ch}(\Psi(\mathcal{E}^{\bullet}_{1}))\cdot ch(\Psi(\mathcal{E}^{\bullet}_{2}))\cdot Td(M^{+}),  \end{equation}
by the Grothendieck-Hirzebruch-Riemann-Roch theorem (see \cite{huy1}), where $\overline{ch}\triangleq\sum_{k}(-1)^{k}ch_{k}$.

\subsection{Pl\"{u}cker type formulas}
Classical Pl\"{u}cker formulas relate degrees, numbers of double points and cusps of dual curves in $\mathbb{P}^{2}$ and its dual space $(\mathbb{P}^{2})^{*}$ \cite{gh}.
The purpose of this subsection is to generalize these formulas to higher dimensions based on equality (\ref{equality of char}) which is deduced from the derived equivalence in Theorem \ref{derived equi of Mukai flops}.
\begin{definition}\label{Legendre dual}
Let $\phi:M\dashrightarrow M^{+}$ be a Mukai flop along $\mathbb{P}^{n}$, and $C$ be a half-dimensional closed irreducible subvariety in $M$.
The \emph{strict transformation} of $C$ is
\begin{equation}C^{\vee}\triangleq\overline{\pi^{+}(\pi^{-1}(C\backslash \mathbb{P}^{n}))}\subseteq M^{+}, \quad \textrm{if}\textrm{ } C\neq\mathbb{P}^{n}, \nonumber \end{equation}
\begin{equation}(\mathbb{P}^{n})^{\vee}\triangleq(-1)^{n}(\mathbb{P}^{n})^{*}\subseteq M^{+}, \nonumber \end{equation}
where $\pi: \widehat{M}\rightarrow M$, $\pi^{+}: \widehat{M}\rightarrow M^{+}$ are the canonical morphisms in (\ref{comm diagram}).
\end{definition}
\begin{remark}
$(C^{\vee})^{\vee}=C$ holds true for irreducible subvarieties.
\end{remark}
A Pl\"{u}cker type formula for strict transformations is given as follows.
\begin{theorem}\label{plucker formula 1}(Pl\"{u}cker type formula for Mukai flop) ${}$ \\
Let $\phi:M\dashrightarrow M^{+}$ be a Mukai flop along $\mathbb{P}^{n}$ ($n\geq2$) between projective manifolds, then
\begin{equation}C_{1}\cdot C_{2}+\frac{(C_{1}\cdot \mathbb{P}^{n})(C_{2}\cdot \mathbb{P}^{n})}{(-1)^{n+1}(n+1)}
=C_{1}^{\vee}\cdot C_{2}^{\vee}+\frac{(C_{1}^{\vee}\cdot (\mathbb{P}^{n})^{*})(C_{2}^{\vee}\cdot (\mathbb{P}^{n})^{*})}{(-1)^{n+1}(n+1)}
\nonumber \end{equation}
holds for $n$-dimensional closed irreducible subvarieties $C_{1}$, $C_{2}$ of $M$ and their strict transformations $C_{1}^{\vee}$, $C_{2}^{\vee}$ in $M^{+}$.
\end{theorem}
\begin{remark}
For half-dimensional subvarieties in $M$ not containing the $\mathbb{P}^{n}$, we can similarly define their strict transformations. Since $(C_{1}\cup C_{2})^{\vee}=C_{1}^{\vee}\cup C_{2}^{\vee}$ holds for such subvarieties, the above formula can be extended to (not necessarily irreducible)
subvarieties in $M$ which do not contain the $\mathbb{P}^{n}$.
\end{remark}
\begin{proof}
We apply (\ref{equality of char}) to $\mathcal{O}_{C_{1}}$, $\mathcal{O}_{C_{2}}$, and then use the following Lemma \ref{chern char of dual lag} to determine Chern characters of $\Psi(\mathcal{O}_{C_{i}})$, $i=1,2$.
\end{proof}
\begin{lemma}\label{chern char of dual lag}
Let $\Psi:D^{b}(M)\rightarrow D^{b}(M^{+})$ be the equivalence in Theorem \ref{derived equi of Mukai flops}, then
\begin{equation}ch(\Psi(\mathcal{O}_{\mathbb{P}^{n}}))=\pm[(\mathbb{P}^{n})^{*}]+ \textrm{h.o.t} ,  \nonumber \end{equation}
\begin{equation}ch(\Psi(\mathcal{O}_{C}))=[C^{\vee}]+\frac{\pm C\cdot\mathbb{P}^{n}-C^{\vee}\cdot(\mathbb{P}^{n})^{*}}{(-1)^{n}(n+1)}[(\mathbb{P}^{n})^{*}]+ \textrm{h.o.t} ,  \nonumber \end{equation}
where $C$ is a $n$-dimensional closed irreducible subvariety in $M$ not containing the $\mathbb{P}^{n}$, $[C]$ denotes its Poincar\'{e} dual, and
$h.o.t$ stands for 'higher order terms'.
\end{lemma}
\begin{proof}
(i) As $M$ and $M^{+}$ are isomorphic outside $\mathbb{P}^{n}$ and $(\mathbb{P}^{n})^{*}$, $supp(\Psi(\mathcal{O}_{\mathbb{P}^{n}}))\subseteq (\mathbb{P}^{n})^{*}$. Then
\begin{equation}ch(\Psi(\mathcal{O}_{\mathbb{P}^{n}}))=\alpha \textrm{ }[(\mathbb{P}^{n})^{*}]+ \textrm{h.o.t}. \nonumber \end{equation}
Applying (\ref{equality of char}) to $\mathcal{O}_{\mathbb{P}^{n}}$ and $\mathcal{O}_{\mathbb{P}^{n}}$, we obtain $\alpha^{2}=1$.

(ii) Away from $(\mathbb{P}^{n})^{*}\subseteq M^{+}$, we have
\begin{equation}supp(\Psi(\mathcal{O}_{C}))\backslash (\mathbb{P}^{n})^{*}=\pi^{+}(\pi^{-1}(C\backslash \mathbb{P}^{n}))=C^{\vee}\backslash (\mathbb{P}^{n})^{*}. \nonumber \end{equation}
Since $supp(\Psi(\mathcal{O}_{C}))$ is closed, it then contains $C^{\vee}$ (with multiplicity one) and
\begin{equation}ch(\Psi(\mathcal{O}_{C}))=[C^{\vee}]+\beta \textrm{ }[(\mathbb{P}^{n})^{*}]+ \textrm{h.o.t}.  \nonumber \end{equation}
Applying (\ref{equality of char}) to $\mathcal{O}_{C}$ and $\mathcal{O}_{\mathbb{P}^{n}}$, we obtain
$\pm C\cdot\mathbb{P}^{n}=C^{\vee}\cdot(\mathbb{P}^{n})^{*}+\beta\cdot(\mathbb{P}^{n})^{*}\cdot(\mathbb{P}^{n})^{*}$.
\end{proof}
For hyper-K\"{a}hler manifold $M$ and its Lagrangian subvariety $C$, the above strict transformation is often referred as Legendre transformation. The Pl\"{u}cker type formula then gives constrains to Legendre dual Lagrangians inside hyper-K\"{a}hler manifolds
\begin{corollary}
Let $M$ be a projective hyper-K\"{a}hler manifold containing a $\mathbb{P}^{n}$ ($n\geq2$), then
\begin{equation}(n+1)(\chi(C)-\chi(C^{\vee}))=(C\cdot \mathbb{P}^{n})^{2}-(C^{\vee}\cdot (\mathbb{P}^{n})^{*})^{2} \nonumber \end{equation}
holds if $C$ is a compact complex Lagrangian submanifold such that $C^{\vee}$ is also smooth.
\end{corollary}
\begin{proof}
Note that $\mathcal{N}_{C/M}\cong T^{*}C$ and $C\cdot C=(-1)^{n}\chi(C)$. Meanwhile, if $C$ is a Lagrangian, its dual is also a Lagrangian (see e.g. \cite{leung1}).
\end{proof}

\subsection{Applications to projective dual varieties}
Given an irreducible subvariety $S$ in $\mathbb{P}^{n}$, there is a notion of projective dual $S^{\vee}$ of $S$ in the dual projective space $(\mathbb{P}^{n})^{*}$ \cite{gkz}. The conormal variety $C_{S}=\overline{\mathcal{N}_{S_{sm}/\mathbb{P}^{n}}^{*}}$ of $S$ is an irreducible complex Lagrangian subvariety in $T^{*}\mathbb{P}^{n}$ whose Legendre transformation is the conormal variety $C_{S^{\vee}}=\overline{\mathcal{N}_{S^{\vee}_{sm}/\mathbb{P}^{n}}^{*}}$ of the projective dual $S^{\vee}$ \cite{leung1}.
As a consequence of Theorem \ref{plucker formula 1}, we obtain a Pl\"{u}cker type formula for dual varieties inside projective spaces.
\begin{theorem}\label{plucker formula 2}
Let $S_{1}$, $S_{2}$ be two closed irreducible subvarieties in $\mathbb{P}^{n}$ ($n\geq2$) which intersect transversally and the same holds true
\footnote{By Proposition 1.3 of \cite{gkz}, $S_{1}^{\vee}$, $S_{2}^{\vee}$ are automatically irreducible.} for their dual varieties $S_{1}^{\vee}$, $S_{2}^{\vee}$ in $(\mathbb{P}^{n})^{*}$. Then we have
\begin{equation}C_{S_{1}}\cdot C_{S_{2}}+\frac{(C_{S_{1}}\cdot \mathbb{P}^{n})(C_{S_{2}}\cdot \mathbb{P}^{n})}{(-1)^{n+1}(n+1)}
=C_{S_{1}^{\vee}}\cdot C_{S_{2}^{\vee}}+\frac{(C_{S_{1}^{\vee}} \cdot (\mathbb{P}^{n})^{*})(C_{S_{2}^{\vee}} \cdot (\mathbb{P}^{n})^{*})}{(-1)^{n+1}(n+1)},
\nonumber \end{equation}
where $C_{S_{i}}$'s are conormal varieties of $S_{i}$'s, $i=1,2$.
\end{theorem}
\begin{proof}
Let $M=\mathbb{P}(\mathcal{O}_{\mathbb{P}^{n}}\oplus T^{*}\mathbb{P}^{n})$ be a compactification of $T^{*}\mathbb{P}^{n}$, and $\overline{C}_{S_{i}}$ be the closure of $C_{S_{i}}$ in $M$. As the intersection of $S_{1}$ and $S_{2}$ is transversal,
$C_{S_{1}}\cap C_{S_{2}}=\overline{C}_{S_{1}}\cap \overline{C}_{S_{2}}\subseteq \mathbb{P}^{n}$ and $C_{S_{i}}\cap \mathbb{P}^{n}=\overline{C}_{S_{i}}\cap \mathbb{P}^{n}$. We interpret products in the above formula using algebraic intersections. By the intersection theory \cite{fulton} (see also the following two propositions), we are then left to prove the formula for $\overline{C}_{S_{i}}\subseteq M$. As topological intersection numbers coincide with algebraic intersection numbers (see for instance, Corollary 19.2 of \cite{fulton}), by Theorem \ref{plucker formula 1}, we are done.
\end{proof}
As the intersection of $S_{1}$ and $S_{2}$ is transversal, the intersection number $C_{S_{1}}\cdot C_{S_{2}}$ depends only on the Euler characteristic of $S_{1}\cap S_{2}$.
\begin{proposition}\label{prop 1.1}
Let $L_{1}$, $L_{2}$ be two irreducible Lagrangian subvarieties inside a holomorphic symplectic manifold $M$ with transversal and compact intersection, then
\begin{equation}L_{1}\cdot L_{2}=(-1)^{dim(L_{1}\cap L_{2})}\chi(L_{1}\cap L_{2}). \nonumber \end{equation}
When applied to Theorem \ref{plucker formula 2}, we have
\begin{equation}C_{S_{1}}\cdot C_{S_{2}}=(-1)^{dim(S_{1}\cap S_{2})}\chi(S_{1}\cap S_{2}). \nonumber \end{equation}
\end{proposition}
\begin{proof}
As $L_{1}\pitchfork L_{2}$, there exists smooth neighbourhoods of $L_{1}\cap L_{2}$ inside both $L_{i}$, $i=1,2$ and we can assume
$L_{1}$, $L_{2}$ are smooth without loss of generality. By the standard excess intersection theory (see \cite{bf2}, \cite{fulton}),
\begin{equation}L_{1}\cdot L_{2}=c_{top}(E)\cap [L_{1}\cap L_{2}], \nonumber \end{equation}
where $E$ is the excess bundle which fits into the exact sequence
\begin{equation}0\rightarrow T(L_{1}\cap L_{2})\rightarrow TL_{1}|_{L_{1}\cap L_{2}}\oplus TL_{2}|_{L_{1}\cap L_{2}}\rightarrow TM|_{L_{1}\cap L_{2}}\rightarrow E\rightarrow 0.     \nonumber \end{equation}
The holomorphic symplectic form on $M$ induces an isomorphism of this sequence to its dual. Thus $E\cong T^{*}(L_{1}\cap L_{2})$.
\end{proof}
The intersection number $(C_{S_{i}}\cdot \mathbb{P}^{n})$ in Theorem \ref{plucker formula 2} is also an intrinsic invariant of $S_{i}$, $i=1,2$.
\begin{proposition}\label{prop 1.2}(MacPherson \cite{mac}, Behrend \cite{behrend}) \\
Let $S$ be a closed irreducible subvariety of $\mathbb{P}^{n}$. We denote $c_{0}^{M}(S)\triangleq(-1)^{dimS}(C_{S}\cdot \mathbb{P}^{n})$. \\
Then it is the degree zero Chern-Mather class of $S$.
Furthermore,
\begin{equation}c_{0}^{M}(S)=\chi(S,Eu([S])), \nonumber \end{equation}
where $Eu$ is the Euler obstruction of $S$ and $\chi(S,Eu([S]))\triangleq \sum_{n\in\mathbb{Z}}n\chi\{Eu([S])=n\}$ is the weighted Euler characteristic with respect to the integer-valued constructible function $Eu([S])$.
\end{proposition}
\begin{remark}
If $S$ is smooth, $Eu(S)=1$ and $c_{0}^{M}(S)=\chi(S)$.
\end{remark}
The following example shows the Pl\"{u}cker type formula for dual varieties is nontrivial.
\begin{example}
Let $S_{1}\cong S_{2}\cong \mathbb{P}^{1}$ be two transversal intersecting lines in $\mathbb{P}^{2}$, and $S_{1}^{\vee}$, $S_{2}^{\vee}$ are two different points in $(\mathbb{P}^{2})^{*}$. The Pl\"{u}cker type formula in Theorem \ref{plucker formula 2} gives
\begin{equation}1-\frac{4}{3}=S_{1}\cap S_{2}+\frac{\chi(\mathbb{P}^{1})\cdot\chi(\mathbb{P}^{1})}{-3}=0+\frac{\chi(pt)\cdot \chi(pt)}{-3}=-\frac{1}{3},  \nonumber \end{equation}
which shows $C_{S_{1}}\cdot C_{S_{2}}\neq C_{S_{1}}^{\vee}\cdot C_{S_{2}}^{\vee}$ in general.
\end{example}
By applying Theorem \ref{plucker formula 2} to the case when $S_{1}=S$ and $S_{2}=\mathbb{P}^{n-k-1}$ with $0\leq k\leq codim(S^{\vee})-1$, we can determine Chern-Mather classes of singular varieties.
\begin{corollary}\label{chern-mather class}
Let $S\subseteq \mathbb{P}^{n}$ be a closed irreducible subvariety such that there exists a linear subspace $\mathbb{P}^{n-k-1}\subseteq\mathbb{P}^{n}$ with
$\mathbb{P}^{n-k-1}\pitchfork S$ and $0\leq k\leq codim(S^{\vee})-1$, then
\begin{equation}c_{0}^{M}(S^{\vee})=(-1)^{dimS+dimS^{\vee}+n+1}\big(\frac{n-k}{k+1}c_{0}^{M}(S)-\frac{n+1}{k+1}\chi(S\cap\mathbb{P}^{n-k-1})\big)    \nonumber \end{equation}
\end{corollary}
\begin{proof}
Note that for generic $\mathbb{P}^{n-k-1}\subseteq\mathbb{P}^{n}$, $(\mathbb{P}^{n-k-1})^{*}\pitchfork S^{\vee}$ and $(\mathbb{P}^{n-k-1})^{*}\cap S^{\vee}=\emptyset$.
\end{proof}
In particular, this recovers classical Pl\"{u}cker formulas for dual curves in $\mathbb{P}^{2}$.
\begin{corollary}(Classical Pl\"{u}cker formula) \\
Let $S_{1}\subseteq \mathbb{P}^{2}$ be an irreducible curve with at most node and cusp singularities and the same holds true for the dual curve $S_{1}^{\vee}$ \footnote{This should be true generically, see \cite{gkz}.}. Suppose $S_{2}\cong \mathbb{P}^{1}$ be a line in $\mathbb{P}^{2}$ intersecting $S_{1}$ transversally\footnote{Such $S_{2}$ always exists as $dim(S^{sing}_{1})< codimS_{2}$.}, then
\begin{equation}c_{0}^{M}(S_{1}^{\vee})=3 deg(S_{1})-2c_{0}^{M}(S_{1}), \nonumber \end{equation}
where $c_{0}^{M}(S_{1})\triangleq(-1)^{dim S_{1}}C_{S_{1}}\cdot \mathbb{P}^{n} =-d^{2}+3d+2\delta+3\kappa$, and $\delta$ (resp. $\kappa$) denotes the number of nodes (resp. cusps) in $S_{1}$.

Moreover, the above formula is equivalent to the classical Pl\"{u}cker formula \footnote{See Chapter 2.4 of \cite{gh} or Proposition 2.5 of \cite{gkz}.}
\begin{equation}d^{\vee}=d^{2}-d-2\delta-3\kappa,  \nonumber \end{equation}
where $d^{\vee}$ denotes the degree of the dual curve $S_{1}^{\vee}$.
\end{corollary}
\begin{proof}
From Corollary \ref{chern-mather class}, $c_{0}^{M}(S_{1}^{\vee})=3deg(S_{1})-2c_{0}^{M}(S_{1})$, then we have
\begin{equation}\label{equality 1}2d^{2}-3d-4\delta-6\kappa=-(d^{\vee})^{2}+3d^{\vee}+2\delta^{\vee}+3\kappa^{\vee},  \end{equation}
\begin{equation}-d^{2}+3d+2\delta+3\kappa=2(d^{\vee})^{2}-3d^{\vee}-4\delta^{\vee}-6\kappa^{\vee}.  \nonumber \end{equation}
Eliminating $(d^{\vee})^{2}$ terms, we obtain $d^{\vee}=d^{2}-d-2\delta-3\kappa$. The converse part is similar.
\end{proof}
\begin{remark}
From the invariance of geometric genus for dual curves, we obtain
\begin{equation}\frac{1}{2}(d-1)(d-2)-\delta-\kappa=\frac{1}{2}(d^{\vee}-1)(d^{\vee}-2)-\delta^{\vee}-\kappa^{\vee}. \nonumber \end{equation}
Combined with (\ref{equality 1}), we get the second Pl\"{u}cker formula
\begin{equation}\kappa^{\vee}=3d(d-2)-6\delta-8\kappa \nonumber \end{equation}
for plane curves (see for instance, Proposition 2.5 \cite{gkz}).
\end{remark}
When $S^{\vee}\cap (\mathbb{P}^{n-k-1})^{*}\neq\emptyset$ in Corollary \ref{chern-mather class}, we have the following example.
\begin{example}(Chern-Mather class of Beauville-Donagi's Pfaffian hypersurface)${}$ \\
Let $X^{13}\subseteq \mathbb{P}^{14}$ be the Pfaffian hypersurface of degree $3$ in Beauville-Donagi \cite{beadonagi},
$\mathbb{P}^{5}\subseteq \mathbb{P}^{14}$ be a generic linear subspace which intersects $X^{13}$ transversally\footnote{The intersection happens inside the smooth loci of $X^{13}$.} along a smooth cubic 4-fold.
Inside the dual space $(\mathbb{P}^{14})^{*}$, we have $(X^{13})^{\vee}\cong Gr(2,6)$, $(\mathbb{P}^{5})^{\vee}\cong \mathbb{P}^{8}$, which intersect transversally along a K3 surface $S=(X^{13})^{\vee}\cap(\mathbb{P}^{5})^{\vee}$. The Pl\"{u}cker type formula in Theorem \ref{plucker formula 2} gives
\begin{equation} \chi(X^{13}\cap\mathbb{P}^{5})+\frac{c_{0}^{M}(X^{13})\cdot \chi(\mathbb{P}^{5})}{-15}=\chi(S)+\frac{\chi(Gr(2,6))\cdot \chi(\mathbb{P}^{8})}{-15},   \nonumber \end{equation}
i.e. $c_{0}^{M}(X^{13})=30$.
\end{example}
By applying Theorem \ref{plucker formula 2} to the case when $S_{1}=S$ and $S_{2}=\mathbb{P}^{n-codim(S^{\vee})-1}$, one can determine the degree of dual varieties, which recovers Ernstr\"{o}m's generalized Pl\"{u}cker formulas (see also \cite{ern1, ern2, kleiman, matsui}).
\begin{corollary}\label{deg of dual var}
Let $S\subseteq \mathbb{P}^{n}$ be a closed irreducible subvariety such that there exists a linear subspace $\mathbb{P}^{n-k-1}\subseteq\mathbb{P}^{n}$ with
$\mathbb{P}^{n-k-1}\pitchfork S$ and $0\leq k\leq codim(S^{\vee})-1$, then
\begin{equation}deg(S^{\vee})=(-1)^{dimS+l+1}\bigg(\frac{l-k}{k+1}c_{0}^{M}(S)+\chi(S\cap \mathbb{P}^{n-l-1})-\frac{l+1}{k+1}\chi(S\cap\mathbb{P}^{n-k-1})\bigg),  \nonumber \end{equation}
where $l=codimS^{\vee}$.
\end{corollary}
\begin{proof}
For generic $\mathbb{P}^{n-k-1}\subseteq\mathbb{P}^{n}$ with $0\leq k\leq codim(S^{\vee})$, $(\mathbb{P}^{n-k-1})^{*}\pitchfork S^{\vee}$,
and the assumption ensures that there exists $\mathbb{P}^{n-codim(S^{\vee})-1}\subseteq\mathbb{P}^{n}$ with $\mathbb{P}^{n-codim(S^{\vee})-1}\pitchfork S$, then we apply Theorem \ref{plucker formula 2} to $(S_{1}=S,S_{2}=\mathbb{P}^{n-k-1})$
and $(S_{1}=S,S_{2}=\mathbb{P}^{n-codim(S^{\vee})-1})$ individually. Combining them and eliminating terms with $C_{S^{\vee}}\cdot (\mathbb{P}^{n})^{*}$, we obtain the formula.
\end{proof}
Combining Corollary \ref{chern-mather class}, \ref{deg of dual var}, we have the following criterion to detect dimensions of dual varieties.
\begin{corollary}\label{dim of dual var}
Let $S\subseteq \mathbb{P}^{n}$ be a closed irreducible subvariety such that there exists a linear subspace $\mathbb{P}^{n-1}\subseteq\mathbb{P}^{n}$ with $\mathbb{P}^{n-1}\pitchfork S$, then for any $0\leq k\leq codim(S^{\vee})-1$,
\begin{equation}kc_{0}^{M}(S)=(k+1)\chi(S\cap\mathbb{P}^{n-1})-\chi(S\cap\mathbb{P}^{n-k-1}),   \nonumber \end{equation}
which becomes a strict inequality exactly when $k=codim(S^{\vee})$.
\end{corollary}
\begin{proof}
When $0\leq k\leq codim(S^{\vee})-1$, we apply Theorem \ref{plucker formula 2} to $(S_{1}=S,S_{2}=\mathbb{P}^{n-1})$ and $(S_{1}=S,S_{2}=\mathbb{P}^{n-k-1})$ and eliminate terms with $c_{0}^{M}(S^{\vee})$.
When $k=codim(S^{\vee})$, we get a strict inequality as $deg(S^{\vee})>0$.
\end{proof}
We apply Theorem \ref{plucker formula 2} to the situation when $S_{2}$ is a smooth quadric hypersurface, in which case its dual variety is again a smooth quadric hypersurface.
\begin{corollary}
Let $S\subseteq \mathbb{P}^{n}$ be a closed irreducible subvariety such that there exists a quadric hypersurface $Q$ intersecting $S$ transversally and the same holds true for $Q^{\vee}$, $S^{\vee}$, then
\begin{equation}\chi(S\cap Q)-\big(1-\frac{1+(-1)^{n}}{2(n+1)}\big)c_{0}^{M}(S)=(-1)^{dimS+dimS^{\vee}}\bigg(\chi(S^{\vee}\cap Q^{\vee})-\big(1-\frac{1+(-1)^{n}}{2(n+1)}\big)c_{0}^{M}(S^{\vee}) \bigg)  \nonumber \end{equation}
\end{corollary}
\begin{proof}
We apply Theorem \ref{plucker formula 2} to $(S_{1}=\mathbb{P}^{n-1},S_{2}=Q)$ and obtain $\chi(Q_{n})=n+1+\frac{1+(-1)^{n}}{2}$ for quadric hypersurface $Q_{n}\subseteq \mathbb{P}^{n+1}$, then apply it to $(S_{1}=S,S_{2}=Q)$.
\end{proof}

\section{Appendix on examples of complex Lagrangians \\ from higher dimensional gauge theories}
The purpose of this section is to construct examples of complex Lagrangian subvarieties inside holomorphic symplectic manifolds following Donaldson-Thomas' work on higher dimensional gauge theories and their TQFT structures (see for instance  \cite{dt}, \cite{tyurin}).

We take a smooth anti-canonical divisor $S$ of a complex projective 3-fold $Y$, and consider a moduli space $\mathfrak{M}_{Y}$ of stable holomorphic bundles on $Y$ with fixed Chern classes. To make sense of the restriction morphism,
\begin{equation}r: \mathfrak{M}_{Y}\rightarrow \mathfrak{M}_{S}   \nonumber \end{equation}
to a moduli space $\mathfrak{M}_{S}$ of stable sheaves on $S$, we recall the following criterion.
\begin{theorem}(Flenner \cite{flenner})\label{restriction thm} ${}$ \\
Let $(X,\mathcal{O}_{X}(1))$ be a complex $n$-dimensional normal projective variety with $\mathcal{O}_{X}(1)$ very ample.
We take $\mathcal{F}$ to be a $\mathcal{O}_{X}(1)$-slope semi-stable torsion-free sheaf of rank $r$. $d$ and $1\leq c\leq n-1$ are integers such that
\begin{equation}[\left(\begin{array}{l}n+d \\ \quad d\end{array}\right)-cd-1]/d>\deg(\mathcal{O}_{X}(1))\cdot\max(\frac{r^{2}-1}{4},1).
\nonumber \end{equation}
Then for a generic complete intersection $Y=H_{1}\cap\cdot\cdot\cdot\cap H_{c}$ with $H_{i}\in |\mathcal{O}_{X}(d)|$,
$\mathcal{F}|_{Y}\triangleq \mathcal{F}\otimes_{\mathcal{O}_{X}} \mathcal{O}_{Y}$ is $\mathcal{O}_{X}(1)|_{Y}$-slope semi-stable on $Y$.
\end{theorem}
\begin{remark}
For $X=\mathbb{P}^{3}$, $\mathcal{O}_{X}(1)=\mathcal{O}_{\mathbb{P}^{3}}(1)$ is very ample. We take $c=1$, $d=5$, then
any semi-stable sheaf on $X$ with $rank\leq5$ remains semi-stable when restricted to a generic quartic $K3$ surface inside $X$.
\end{remark}
Assuming conditions in Theorem \ref{restriction thm} are satisfied, we obtain a morphism
\begin{equation}r: \mathfrak{M}_{Y}\rightarrow \mathfrak{M}_{S}, \nonumber \end{equation}
whose deformation obstruction theory is described by the following exact sequence.
\begin{lemma}\label{def-obs LES}
Let $E\in\mathfrak{M}_{Y}$ be a stable bundle, and $S$ be connected. Then there is a long exact sequence,
\begin{equation}0\rightarrow H^{0,1}(Y,EndE\otimes K_{Y})\rightarrow H^{0,1}(Y,EndE)\rightarrow H^{0,1}(S,EndE|_{S})
\rightarrow  \nonumber \end{equation}
\begin{equation}\rightarrow H^{0,2}(Y,EndE\otimes K_{Y})\rightarrow  H^{0,2}(Y,EndE)\rightarrow 0.\nonumber \end{equation}
\end{lemma}
\begin{proof}
We tensor $0\rightarrow\mathcal{O}_{Y}(-S)\rightarrow\mathcal{O}_{Y}\rightarrow\mathcal{O}_{S}\rightarrow0$ with $EndE$ and take its cohomology.
\end{proof}
It was observed by Donaldson-Thomas \cite{dt} (in fact first by Tyurin \cite{tyurin}) that the transpose of the above sequence with respect to Serre duality pairings on $Y$ and $S$ remains the same, and $r(\mathfrak{M}_{Y})$ will be a complex Lagrangian submanifold of $\mathfrak{M}_{S}$ provided that $H^{0,2}(Y,EndE)=0$.
\begin{example}(Li-Qin's example, Proposition 5.4 of \cite{lq})

Let $S$ be a ($K3$) generic hyperplane section of type $(2,3)$ in $Y=\mathbb{P}^{1}\times \mathbb{P}^{2}$, and $L_{m}=\mathcal{O}_{Y}(1,m)$ be a polarization. Then the moduli space $\mathfrak{M}_{Y}$ of rank $2$, $L_{m}$-stable bundles on $Y$ with Chern classes $c_{1}=(\epsilon_{1},\epsilon_{2})$, $c_{2}=(-1,1)\cdot(\epsilon_{1}+1,\epsilon_{2}-1)$ is isomorphic
to $\mathbb{P}^{k}$, where $(\epsilon_{1},\epsilon_{2})=(0,1)$ or $(1,0)$ or $(1,1)$, and $k=(5+6\epsilon_{1}-3\epsilon_{2}-3\epsilon_{1}\epsilon_{2})$ respectively.

Furthermore, if $\frac{2(2-\epsilon_{2})}{2+\epsilon_{1}}<m<\frac{2(2-\epsilon_{2})}{\epsilon_{1}}$, the restriction map
\begin{equation}\mathfrak{M}_{Y}\hookrightarrow \mathfrak{M}_{S}  \nonumber \end{equation}
to a moduli space $\mathfrak{M}_{S}$ of $L_{m}|_{S}$-(semi)stable rank $2$ torsion-free sheaves
on $S$ with Chern classes $c_{1}|_{S}$, $c_{2}|_{S}$ is an imbedding of a complex Lagrangian into a compact hyper-K\"{a}hler manifold.

In particular, if $(\epsilon_{1},\epsilon_{2})=(0,1)$, for $m\geq2$, we have $\mathfrak{M}_{Y}\cong\mathbb{P}^{2}$ which is \textbf{not spin} and is not an \emph{orientable complex Lagrangian} based on Joyce's definition (see e.g. \cite{bbdjs}, \cite{joyce} or Definition 1.16 of \cite{bussi}). However, one could still categorify complex Lagrangian intersection $\mathfrak{M}_{Y}\cap \mathfrak{M}_{Y}\subseteq \mathfrak{M}_{S}$, where the cohomology theory is $H^{*}(\mathfrak{M}_{Y},\mathbb{C})$ (see e.g. \cite{bf2}).
\end{example}

The main difficulty of extending the above construction to general cases is that the restriction morphism is not well-defined for general stable sheaves. However, for ideal sheaves of subschemes, say $I_{Z}$'s, the restriction map is well-defined if $Z$'s are normal to the divisor $S\subseteq Y$, i.e. $Tor^{1}(\mathcal{O}_{Z},\mathcal{O}_{S})=0$ \footnote{See the work of Li and Wu \cite{liwu}, \cite{wu}.}. Then such ideal sheaves form an open subspace $U$ of the moduli scheme of ideal sheaves on $Y$ which has a restriction morphism $r: U\rightarrow Hilb^{*}(S)$ to a Hilbert scheme of points on $S$.
Then it is interesting to know when $\overline{Im(r)}\subseteq Hilb^{*}(S)$ is a complex Lagrangian subvariety.
\begin{proposition}\label{prop 2}
Let $L$ be a connected smooth complex quasi-projective variety, $M$ be an algebraic symplectic manifold, and $r: L\rightarrow M$ be an algebraic morphism whose underlying complex analytic map is an imbedding of a complex Lagrangian submanifold. Then the (Zariski) closure $\overline{Im(r)}$ in $M$ is a complex Lagrangian subvariety.
\end{proposition}
\begin{proof}
The morphism $r$ factors through a morphism $r: L\rightarrow Im(r)_{sch}$ to the scheme theoretic image of $r$, which is the 'smallest' closed subset of $M$ containing the set $Im(r)$. As $L$ is reduced, the scheme theoretic image of $r$ coincides with the closure $\overline{Im(r)}$ (with reduced structures) of the image of $r$ (see e.g. 8.3.A of \cite{vakil}). As $L$ and $\overline{Im(r)}$ are both algebraic (they are closed subvarieties of algebraic varieties), we can find a subset $U\subseteq Im(r)$ which is dense and open in $\overline{Im(r)}$ by Chevalley's lemma (see e.g. Lemma 2 of \cite{serre}). Then we have a morphism $r: r^{-1}(U)\rightarrow U$ whose underlying complex analytic map is an isomorphism. By Serre's GAGA principle (Proposition 9 of \cite{serre}), it is also an isomorphism between algebraic schemes. The connectedness of $L$ implies that $r^{-1}(U)$, $U$ are both irreducible, so is $\overline{Im(r)}$. Since $U$ is an irreducible locally closed complex Lagrangian submanifold of $M$, by Proposition \ref{prop}, $\overline{U}=\overline{Im(r)}\subseteq M$ is a complex Lagrangian subvariety.
\end{proof}
We apply the above proposition to construct compact complex Lagrangian subvarieties.
\begin{example}(Generic quartics in $\mathbb{P}^{3}$)

Let $S$ be a generic quartic surface in $Y=\mathbb{P}^{3}$ as its anti-canonical divisor. \\
(i) We take the primitive curve class $[H]\in H_{2}(Y,\mathbb{Z})$. Ideal sheaves of curves representing this class have Chern character $c=(1,0,-PD([H]),1)$. We consider the moduli space $I_{1}(Y,[H])$ of such ideal sheaves on $Y$ and $I_{1}(Y,[H])\cong Gr(2,4)$. As $S$ is a projective $K3$ surface and contains only one primitive rational curve, denoted by $C_{0}$ \cite{yauzaslow, bea, bleung, chenxi}, we obtain a well-defined restriction morphism
\begin{equation}r: I_{1}(Y,[H])\backslash \{I_{C_{0}}\} \rightarrow Hilb^{4}(S),  \nonumber \end{equation}
\begin{equation}r: I_{C}\mapsto I_{C}|_{S},  \nonumber \end{equation}
which is injective with smooth image. By direct calculations, for any $I_{C}\in I_{1}(Y,[H])$, we have
\begin{equation}Ext^{1}_{Y}(I_{C},I_{C})\cong \mathbb{C}^{4}, \quad Ext^{i\geq2}_{Y}(I_{C},I_{C})=0. \nonumber \end{equation}
Then the injective restriction map determines a complex Lagrangian submanifold
\begin{equation}Gr(2,4)\backslash \{pt\}\subseteq Hilb^{4}(S). \nonumber \end{equation}
Its closure $\overline{Im(r)}\subseteq Hilb^{4}(S)$ is a complex Lagrangian subvariety by Proposition \ref{prop 2}.  \\
${}$ \\
(ii) We take the degree $2$ curve class $[2H]\in H_{2}(Y,\mathbb{Z})$, and ideal sheaves of curves representing this class have Chern character $c=(1,0,-PD([2H]),3)$. We consider the moduli space $I_{3}(Y,[2H])$ of such ideal sheaves, and $I_{3}(Y,[2H])$ is a $\mathbb{P}^{5}$-bundle over $\mathbb{P}^{3}$ \footnote{See Remark 2.2.7 of \cite{deland} for reference.}. As $S$ contains only one rational curve of degree $2$
\cite{yauzaslow, leeleung, wu0, kmps, pt1} similarly as before, we have a well-defined restriction morphism
\begin{equation}r: I_{3}(Y,[2H])\backslash \{pt\} \rightarrow Hilb^{8}(S),  \nonumber \end{equation}
\begin{equation}r: I_{C}\mapsto I_{C}|_{S},  \nonumber \end{equation}
which is injective with smooth image. By direct calculations, for any $I_{C}\in I_{3}(Y,[2H])$, we have
\begin{equation}Ext^{1}_{Y}(I_{C},I_{C})\cong \mathbb{C}^{8}, \quad Ext^{i\geq2}_{Y}(I_{C},I_{C})=0. \nonumber \end{equation}
Then the injective restriction map determines a complex Lagrangian submanifold
\begin{equation}I_{3}(Y,[2H])\backslash \{pt\} \subseteq Hilb^{8}(S) \nonumber \end{equation}
and a compact complex Lagrangian subvariety $\overline{Im(r)}\subseteq Hilb^{8}(S)$.
\end{example}
\begin{remark}
One could consider $K3$ surfaces as anti-canonical divisors of other projective 3-folds and give more examples of complex Lagrangians inside hyper-K\"{a}hler manifolds.

For instance, we take a generic degree $1\leq d\leq 4$ hypersurface $Y\subseteq \mathbb{P}^{4}$ and consider the primitive class $[H]\in H_{2}(Y,\mathbb{Z})$. Ideal sheaves of curves representing this class have Chern character $c=(1,0,-PD([H]),\frac{3-d}{2})$, and we denote their moduli space by $I_{\frac{3-d}{2}}(Y,[H])$. By Theorem 4.3 in Chapter V of \cite{kollar}, $I_{\frac{3-d}{2}}(Y,[H])$ is a smooth connected projective variety of dimension $(5-d)$. By Chapter V, 4.4 of \cite{kollar} and direct calculations, for any $I_{C}\in I_{\frac{3-d}{2}}(Y,[H])$,
\begin{equation}Ext^{1}_{Y}(I_{C},I_{C})\cong\mathbb{C}^{5-d}, \quad Ext^{i\geq2}_{Y}(I_{C},I_{C})=0. \nonumber \end{equation}
Then the restriction morphism to the Hilbert scheme of a generic anti-canonical divisor $S\subseteq Y$,
\begin{equation}r: I_{\frac{3-d}{2}}(Y,[H])\backslash \{pt\} \rightarrow Hilb^{5-d}(S),  \nonumber \end{equation}
\begin{equation}r: I_{C}\mapsto I_{C}|_{S},  \nonumber \end{equation}
determines a complex Lagrangian submanifold
\begin{equation}I_{\frac{3-d}{2}}(Y,[H])\backslash \{pt\}\subseteq Hilb^{5-d}(S) \nonumber \end{equation}
and a compact complex Lagrangian subvariety $\overline{Im(r)}\subseteq Hilb^{5-d}(S)$.
\end{remark}

\end{document}